\documentclass[fleqn,reqno,11pt,a4paper,final]{amsart}

\usepackage[a4paper,left=30mm,right=30mm,top=30mm,bottom=30mm,marginpar=20mm]{geometry}
\usepackage{amsmath}
\usepackage{amssymb}
\usepackage{amsthm}
\usepackage{amscd}
\usepackage[ansinew]{inputenc}
\usepackage{cite}
\usepackage{bbm}
\usepackage{color}
\usepackage[english=american]{csquotes}
\usepackage[final]{graphicx}
\usepackage{hyperref}
\usepackage{calc}
\usepackage{mathptmx}

\graphicspath{{../Pictures/}}

\numberwithin{equation}{section}

\newtheoremstyle{thmlemcorr}{10pt}{10pt}{\itshape}{}{\bfseries}{.}{10pt}{{\thmname{#1}\thmnumber{ #2}\thmnote{ (#3)}}}
\newtheoremstyle{thmlemcorr*}{10pt}{10pt}{\itshape}{}{\bfseries}{.}\newline{{\thmname{#1}\thmnumber{ #2}\thmnote{ (#3)}}}
\newtheoremstyle{remexample}{10pt}{10pt}{}{}{\bfseries}{.}{10pt}{{\thmname{#1}\thmnumber{ #2}\thmnote{ (#3)}}}
\newtheoremstyle{ass}{10pt}{10pt}{}{}{\bfseries}{.}{10pt}{{\thmname{#1}\thmnumber{ A#2}\thmnote{ (#3)}}}

\theoremstyle{thmlemcorr}
\newtheorem{theorem}{Theorem}
\numberwithin{theorem}{section}

\newtheorem{definition}[theorem]{Definition}

\theoremstyle{thmlemcorr*}
\newtheorem{theorem*}{Theorem}
\newtheorem{lemma*}[theorem]{Lemma}
\newtheorem{corollary*}[theorem]{Corollary}
\newtheorem{proposition*}[theorem]{Proposition}
\newtheorem{problem*}[theorem]{Problem}
\newtheorem{conjecture*}[theorem]{Conjecture}
\newtheorem{definition*}[theorem]{Definition}

\theoremstyle{remexample}
\newtheorem{remark}[theorem]{Remark}

\theoremstyle{ass}

\newcommand{\T}{\mathbb{T}}

\DeclareMathOperator{\diverg}{div}

\DeclareMathOperator{\trace}{trace}

\newcommand{\norm}[1]{\|#1\|}

\newcommand{\N}{\mathbb{N}}
\newcommand{\R}{\mathbb{R}}

\newcommand{\To}{\mathbb{T}}

\newcommand{\term}[1]{\textbf{#1}}

\def\XXint#1#2#3{{\setbox0=\hbox{$#1{#2#3}{\int}$}
\vcenter{\hbox{$#2#3$}}\kern-.5\wd0}}

\renewcommand{\epsilon}{\varepsilon}
\renewcommand{\phi}{\varphi}

\begin{document}


\title[Weak-Strong Uniqueness]{Weak-Strong Uniqueness in Fluid Dynamics}

\author{Emil Wiedemann}
\address{\textit{Emil Wiedemann:} Institute of Applied Mathematics, Leibniz University Hannover, Welfengarten~1, 30167 Hannover, Germany}
\email{wiedemann@ifam.uni-hannover.de}

\begin{abstract}
We give a survey of recent results on weak-strong uniqueness for compressible and incompressible Euler and Navier-Stokes equations, and also make some new observations. The importance of the weak-strong uniqueness principle stems, on the one hand, from the instances of non-uniqueness for the Euler equations exhibited in the past years; and on the other hand from the question of convergence of singular limits, for which weak-strong uniqueness represents an elegant tool.  
\end{abstract}







\maketitle




\section{Introduction}
The Euler and Navier-Stokes equations can be viewed as the fundamental partial differential equations of the continuum mechanics of fluids. Known for centuries, their rigorous mathematical treatment has however remained vastly incomplete. For a system of time-dependent differential equations supposed to describe classical mechanics, one would expect existence and uniqueness of solutions given suitable initial (and possibly boundary) conditions. In addition, in any real situation the data is necessarily subject to measurement errors, and so one should require small initial perturbations to cause only small perturbations at later times. These three properties -- existence, uniqueness, and stability of solutions -- are known as the \term{well-posedness} of a partial differential equation. Of course, to make sense of the notion of well-posedness, one has to make various specifications of the problem, like the function spaces one works in.

In the realm of fluid dynamics, one often talks about \term{laminar} versus \term{turbulent} flows, the former being characterised by a smooth, predictable dynamics of the fluid, while the latter features chaotic and irregular behaviour. We will not become more precise about this distinction, but the reader should imagine a laminar flow when we talk about a ``smooth" or ``strong" solution (in a class like $C^1$ or better), and a turbulent one when we mention ``weak" solutions (which could be as bad as a generic $L^2$ function).

For both the Euler and the Navier-Stokes equations (compressible or incompressible), in the physically most relevant case of three space dimensions, well-posedness is known only locally in time and in function spaces of high regularity (specifically, $C^{1,\alpha}$ or Sobolev spaces that embed into $C^{1,\alpha}$). We refer the reader to the well-known textbooks in the field, e.g.~\cite{constantinfoias, temam, bertozzimajda, marchioropulvirenti, lions, lions2, feireislbook, galdinotes, robinsonetal}. Thus, a laminar flow remains laminar at least for a short time. Whether or not blow-up can actually occur after a finite time is in fact one of the Clay Foundation's Millennium Problems.

On the other hand, it is known that well-posedness can fail for the three-dimensional Euler equations in spaces of lower regularity, like H\"older spaces $C^\alpha$ \cite{bardostiti} or critical Sobolev spaces \cite{bourgainli}. This means that a flow initially belonging to such a class may drop out of this class instantaneously.

More shockingly, at lower regularity one may observe \term{non-uniqueness}: Two distinct solutions of the same equation can arise from the same initial data. This was shown for the first time by Scheffer \cite{scheffer} (see also \cite{shnirel1}), who constructed a nontrivial weak solution of the 2D incompressible Euler equations (in $L^2$) with compact support in time. More recently, De Lellis and Sz\'ekelyhidi recovered and extended these examples in a series of groundbreaking papers (among them \cite{euler1, inventiones}). Their adaptation of Gromov's technique of \term{convex integration} led to various important new results, like non-uniqueness of weak solutions satisfying additional energy (in)equalities \cite{euler2, euleryoung, danerisz}, global existence and non-uniqueness for arbitrary initial data in $L^2$ \cite{eulerexistence}, non-uniqueness of entropy solutions for the isentropic compressible Euler equations \cite{euler2, chiodaroli, chiodarolidelelliskreml}, and, most recently, the proof of Onsager's Conjecture on energy-dissipating H\"older continuous solutions \cite{isett16, buckmasteretal17}. It remains open however if weak solutions of the Euler equations exist for all finite energy data.

For the Navier-Stokes equations, at least the global existence of weak solutions is known (Leray \cite{leray} for the incompressible and Lions \cite{lions2} and Feireisl-Novotn\'y-Petzeltov\'a \cite{feireislnovotnypetzeltova} for the compressible case), but their uniqueness remains an open problem. An interesting indication of non-uniqueness has been demonstrated recently \cite{sverak}.

In the light of these results, is there no hope to obtain unique weak solutions? There are two conceivable remedies: Either one imposes additional admissibility criteria to single out a unique solution, or one contends oneself with a conditional form of uniqueness. The first strategy is still at a very speculative stage and does not currently promise much success. In addition, it may be objected that the concept of a deterministic solution is rather meaningless in the turbulent regime, where any uniquely determined solution may be so unstable that it would be practically useless.

The second strategy is \term{weak-strong uniqueness} and was established by Leray \cite{leray}, Prodi \cite{prodi}, and Serrin \cite{serrin} for the incompressible Navier-Stokes equations, and by Dafermos \cite{Daf} for conservation laws. The general principle can be stated as follows:

\emph{If there exists a strong solution, then any weak solution with the same initial data coincides with it.}

Note carefully that is not merely a uniqueness result of the standard form ``If two solutions of the same type share the same initial data, then they are identical", as encountered in classical well-posedness theories. Rather, the strong solution is unique even when compared to all \emph{weak} solutions. We may thus slightly reformulate the weak-strong uniqueness principle:

\emph{Strong solutions are unique in the potentially much larger class of weak solutions.}

The meaning of ``strong" and ``weak" will be made precise in each particular situation. In any case, it is crucial that the notion of weak solution include some sort of \term{energy inequality}: If this ingredient is removed, then already the Scheffer-Shnirelman construction yields a counterexample to weak-strong uniqueness. Another restriction is that for \emph{inviscid} models we need to work on domains without boundaries, see Section \ref{boundary} below. 

For the moment, however, note that weak-strong uniqueness is only helpful in the laminar regime, as it presupposes the existence of a strong (``laminar") solution. Very roughly, one could say that weak-strong uniqueness tells that the laminar and turbulent regimes are in a sense disjoint: \emph{Either} the flow is laminar and represented by a unique strong solution, \emph{or} it is turbulent and represented by possibly very many weak solutions.

There is another reason for the recently refreshed interest in weak-strong uniqueness: The principle can come in very handy for the convergence of singular limits, such as vanishing viscosity, hydrodynamic limits, or numerical approximation schemes. The generic question, which is very classical, can be stated as follows:

\emph{Suppose the limit system has a strong solution. Does the singular limit converge to this solution?}  

Suppose, for instance, one considers the incompressible Euler equations with smooth initial data, so that by classical well-posedness results there exists a smooth solution at least for a certain time period. It is arguably rather difficult to show directly that the solutions of the Navier-Stokes equations converge to this solution as the viscosity tends to zero\footnote{For the incompressible Navier-Stokes-to-Euler limit without boundaries, there are several classical results on this question. However the question may be much harder for other systems of fluid mechanics.}. But it may be much easier to show that the Navier-Stokes solutions converge to some weak kind of solution to the Euler equations. Then weak-strong uniqueness kicks in: Indeed, the weak solution will automatically be the smooth one, and thus convergence of the viscosity limit to the smooth solution of Euler is shown.

Unfortunately, it is presently not even possible to show convergence of a zero viscosity sequence to a weak solution of the Euler equations in the sense of distributions. Instead, one has to relax even further the notion of solution. This was done by P.-L.\ Lions \cite{lions}, who essentially built the weak-strong uniqueness property into his definition of \term{dissipative solutions}.\footnote{The terminology has become somewhat ambiguous, as ``dissipative solution" has recently often been used to refer to a solution in the sense of distributions that satisfies some form of the energy inequality, or to a solution which satisfies such an inequality in a strict way.}  It is easy to show the convergence of a vanishing viscosity sequence to such a dissipative solution and thus to the strong solution, if the latter exists. For applications to hydrodynamic limits, see \cite{saintraymond} and references therein. 

We prefer here to work with so-called \term{measure-valued solutions}, as introduced by DiPerna and Majda \cite{dipernamajda} for the incompressible Euler equations. This notion of solution also enjoys the weak-strong uniqueness property (as shown in \cite{brenieretal}) and can therefore be considered ``equivalent" to Lions' dissipative solutions. Weak-strong uniqueness for measure-valued solutions has recently been proved for various other equations of fluid dynamics \cite{GwSwWi, feireislgwiazdawsu, brezinamacha, brezinafeireisl} and applied to the viscosity limit of nonlocal Euler equations \cite{brezinamacha} and a finite element/finite volume scheme for the compressible Navier-Stokes equations \cite{feireisllukacova}. The latter result in particular shows that the weak-strong uniqueness principle can be of concrete practical use in numerical analysis (and certainly much more is to be explored in this direction).      

The aim of this paper is to give a survey over the main ideas in the topic, rather than a complete overview of all available results. There are however two original contributions: In Section \ref{mvs}, we give a slight simplification and generalisation of the result of Brenier et al.\ \cite{brenieretal}, adapting their arguments to the framework of \cite{feireislgwiazdawsu}; and in Section \ref{alternative} we present an alternative proof of weak-strong uniqueness for weak solutions of the incompressible Euler equations, inspired by the notion of statistical solutions in the sense of Fjordholm et al.\ \cite{fjordholmetal}.

The article is organised as follows: In Section \ref{relenergy}, we introduce the relative energy method, which is at the heart of all our arguments, for the simplest case of the incompressible Euler equations, and also give the calculation for the isentropic Euler system. Section \ref{mvs} presents the measure-valued framework in the formalism of \cite{feireislgwiazdawsu} and applies this setup to the incompressible Euler equations. In Section \ref{viscosity} we discuss the incompressible and compressible Navier-Stokes equations, noting that the techniques of Sections \ref{mvs} and \ref{viscosity} can be combined to yield weak-strong uniqueness for measure-valued solutions of the compressible Navier-Stokes equations as in \cite{feireislgwiazdawsu}. Our treatment of the viscous term in the compressible case is based on \cite{FeJiNo}. For Euler models, physical boundaries pose an additional difficulty, as weak-strong uniqueness may then fail. We discuss a possible remedy in Section \ref{boundary}. Section \ref{alternative} then gives an alternative proof of weak-strong uniqueness for the incompressible Euler equations. 
              
\term{Acknowledgements.} The author would like to thank Eduard Feireisl, Ulrik Skre Fjordholm, and Siddharta Mishra for very helpful discussions on the topic.

\section{The Relative Energy Method}\label{relenergy}
In this section we demonstrate the relative energy method, which forms the basis of all proofs of weak-strong uniqueness, in the simplest case of the incompressible Euler equations, and then in the slightly more tedious case of the isentropic compressible Euler equations, on the torus. Recall that the incompressible Euler equations are given as
\begin{equation}\label{euler}
\begin{aligned}
\partial_t u + \diverg(u\otimes u)+\nabla p&=0,\\
\diverg u&=0,
\end{aligned}
\end{equation}
and a vector field $u\in L^2(\T^d\times [0,T])$ is called a \term{weak solution} of this system with initial datum $u^0$ if, for almost every $\tau\in (0,T)$ and every divergence-free vector field $\phi\in C^1(\T^d\times [0,T];\R^d)$, we have  
\begin{equation*}
\begin{aligned}
&\int_0^\tau\int_{\T^d}\partial_t\phi(x,t)\cdot u(x,t)+\nabla\phi(x,t):(u\otimes u)(x,t)dxdt\\
&\hspace{3cm}=\int_{\T^d}u(x,\tau)\cdot\phi(x,\tau)-u^0(x)\cdot\phi(x,0)dx,
\end{aligned}
\end{equation*}
and if $u$ is weakly divergence-free in the sense that
\begin{equation*}
\int_{\T^d}u(x,\tau)\cdot\nabla \psi(x)dx=0
\end{equation*}
for almost every $\tau\in(0,T)$ and every $\psi\in C^1(\T^d)$. Note that the pressure disappeared from the equations in the weak formulation, as it is merely a Lagrange multiplier. Once a weak solution $u$ has been found, one can find a suitable pressure by solving the Poisson equation
\begin{equation*}
-\Delta p=\diverg\diverg(u\otimes u).
\end{equation*}
We call a weak solution $u$ \term{globally admissible}\footnote{Global admissibility is a weaker condition than the local variant of the energy inequality, which resembles the entropy condition in the theory of conservation laws.} if $u\in L^\infty((0,T);L^2(\T^d))$, and
\begin{equation*}
\frac{1}{2}\int_{\To^d}|u(x,\tau)|^2dx\leq \frac{1}{2}\int_{\To^3}|v^0(x)|^2dx
\end{equation*} 
for almost every $\tau\in (0,T)$. In fact one can show (see the Appendix in \cite{euler2}) that a globally admissible weak solution can be altered on a set of times of measure zero so that it becomes continuous from $\tau$ into $L^2(\T^d)$ equipped with the weak topology. In particular $u(\cdot,\tau)$ is well-defined for \emph{every} time $\tau\in[0,T]$, and the initial datum is the weak limit of $u(\cdot,\tau)$ as $\tau\searrow 0$.

We can now state and prove our prototypical weak-strong uniqueness result:
\begin{theorem}\label{basicthm}
Let $u\in L^\infty((0,T);L^2(\mathbb{T}^d))$ be a weak solution and $U\in C^1(\mathbb{T}^d\times[0,T])$ a strong solution of \eqref{euler}, and assume that $u$ and $U$ share the same initial datum $u^0$. Assume moreover that
\begin{equation*}
\frac{1}{2}\int_{\T^d}|u(x,\tau)|^2dx\leq\frac{1}{2}\int_{\T^d}|u^0(x)|^2dx
\end{equation*}
for almost every $\tau\in(0,T)$. Then $u(x,\tau) =U(x,\tau)$ for almost every $(x,\tau)\in\To^d\times(0,T)$.
\end{theorem}
\begin{proof}
We define for almost every $\tau\in(0,T)$ the \term{relative energy} $E_{rel}$ and compute 
\begin{align*}
E_{rel}(\tau)&:= \frac{1}{2}\int_{\To^d}|u(x,\tau)-U(x,\tau)|^2dx\\
&=\frac{1}{2}\int_{\To^3}|U(x,\tau)|^2dx+\frac{1}{2}\int_{\To^d}|u(x,\tau)|^2dx-\int_{\To^d}u(x,\tau)\cdot U(x,\tau)dx\\
&\leq \frac{1}{2}\int_{\To^3}|u^0(x)|^2dx+\frac{1}{2}\int_{\To^d}|u^0(x)|^2dx-\int_{\To^d}u(x,\tau)\cdot U(x,\tau)dx\\
&=\int_{\To^3}|u^0(x)|^2dx-\int_{\To^3}u^0(x)\cdot u^0(x)dx\\
&\hspace{1cm}-\int_0^\tau\int_{\To^d}[\partial_tU(x,t)\cdot u(x,t)+\nabla U(x,t):(u\otimes u)(x,t)]dxdt\\
&=-\int_0^\tau\int_{\To^d}[\partial_tU(x,t)\cdot u(x,t)+\nabla U(x,t):(u\otimes u)(x,t)]dxdt\\
&=\int_0^\tau\int_{\To^d}[\diverg(U\otimes U)(x,t)\cdot u(x,t)-\nabla U(x,t):(u\otimes u)(x,t)]dxdt\\
&=\int_0^\tau\int_{\To^d}[(U-u)(x,t)\cdot\nabla_{sym} U(x,t)(u-U)(x,t)]dxdt\\
&\leq \int_0^\tau\norm{\nabla_{sym}U(t)}_{L^\infty(\To^d)}\int_{\To^d}|u-U|^2(x,t)dxdt\\
&= 2\int_0^\tau\norm{\nabla_{sym}U(t)}_{L^\infty(\To^d)}E_{rel}(t)dt,
\end{align*}
and the claim follows from Gr\"onwall's inequality. In this calculation, we used the weak energy inequality for $u$, the energy equality for $U$, the definition of $u$ being a weak solution, the fact that $U$ is a strong solution, and the identities 
\begin{equation*}
\diverg(U\otimes U)=\nabla UU,
\end{equation*} 
\begin{equation*}
\nabla U:(u\otimes u)=u\cdot\nabla U u,
\end{equation*}
and
\begin{equation*}
\int_{\To^d}(U-u)\cdot\nabla U Udx=0,
\end{equation*}
the first and third of these equalities making use of the divergence-free property of $U$ and $u$. Note also that $(x,Ax)=(x,A_{sym}x)$ for any $x\in\R^d$ and $A\in\R^{d\times d}$, so that indeed it suffices to consider the symmetric gradient.
\end{proof}

\begin{remark}
The Gr\"onwall estimate suggests that the property $\nabla_{sym} U\in L^1_tL^\infty_x$ suffices for the strong solution, and we do not need to assume $U\in C^1_{t,x}$. This is probably true (and in fact claimed in \cite{brenieretal}), but requires some further approximation arguments that the author considers non-trivial. We will not pursue this problem here.
\end{remark}

Let us demonstrate the usefulness of the relative energy method for a somewhat more complicated fluid model, the \term{isentropic compressible Euler equations} with adiabatic exponent $\gamma>1$:
\begin{equation}\label{isentropic}
\begin{aligned}
\partial_t (\rho u) + \diverg(\rho u\otimes u)+\nabla \rho^\gamma&=0,\\
\partial_t\rho+\diverg (\rho u)&=0.
\end{aligned}
\end{equation}

Again we can make sense of these equations in the sense of distributions: A pair $(\rho,u)$ is a \term{weak solution} of the isentropic system with initial datum $(\rho^0,u^0)$ if $\rho\in L^\gamma(\T^d\times(0,T))$, $\rho|u|^2\in L^1(\T^d\times (0,T))$, and

\begin{equation}\label{Emass_momentum}
\begin{aligned}
\int_0^\tau\int_{\T^d}\partial_t\psi \rho+\nabla\psi\cdot{\rho u}dxdt+\int_{\T^d}\psi(x,0)\rho^0-\psi(x,\tau)\rho(x,\tau)dx&=0,\\
\int_0^\tau\int_{\T^d}\partial_t\phi\cdot{\rho u}+\nabla\phi : {(\rho u\otimes u)}+\diverg\phi{\rho^\gamma}dxdt\\
+\int_{\T^d}\phi(x,0)\cdot \rho^0u^0-\phi(x,\tau)\cdot{\rho u}(x,\tau)dx&=0
\end{aligned}
\end{equation}
for almost every $\tau\in(0,T)$, every $\psi\in C^1(\T^d\times[0,T])$, and every $\phi\in C^1(\T^d\times[0,T];\R^d)$. In analogy with the incompressible situation, we say the solution is \term{globally admissible} if for almost every $\tau\in(0,T)$
\begin{equation}\label{Emvsenergy}
E(\tau)\leq E^0,
\end{equation}
where the energy $E$ is defined as
\begin{equation*}
E(\tau)=\frac{1}{2}\int_{\T^d}\rho(x,\tau)|u(x,\tau)|^2dx+\frac{1}{\gamma-1}\rho(x,\tau)^\gamma dx,
\end{equation*}
and
\begin{equation*}
E^0=\frac{1}{2}\int_{\T^d}\rho^0(x)|u^0(x)|^2dx+\frac{1}{\gamma-1}\rho^0(x)^\gamma dx.
\end{equation*}
Note that the global admissibility condition is weaker than the local energy equality (the so-called \term{entropy condition}) usually invoked in the study of compressible Euler systems.

 As an analogue of the preceding theorem, we have
\begin{theorem}\label{Eweak-strong}
Let $R,U\in C^1(\To^d\times[0,T])$ be a solution of \eqref{isentropic} with initial data $\rho^0,u^0$ such that $\rho^0\geq c>0$ and $R\geq c>0$. If $(\rho,u)$ is an admissible solution with the same initial data, then $\rho\equiv R$ and $u\equiv U$ almost everywhere on $\To^d\times(0,T)$. 
\end{theorem}

\begin{proof}
Let us first define for a.e.\ $\tau\in[0,T]$ the \term{relative energy} between $(R,U)$ and $(\rho,u)$ as
\begin{equation*}
E_{rel}(\tau)=\int_{\To^d}\frac{1}{2}{\rho|u-U|^2}+{\frac{1}{\gamma-1}\rho^\gamma-\frac{\gamma}{\gamma-1}R^{\gamma-1}\rho+R^\gamma}dx.
\end{equation*}
Note that, as $\gamma>1$, the map $|\cdot|^\gamma$ is strictly convex, which implies that the relative energy is always non-negative with equality if and only if $\rho\equiv R$ and $u\equiv U$. Thus $E_{rel}(\tau)=0$ for a.e.\ $\tau$ implies Theorem~\ref{Eweak-strong}. 

We set $\phi=U$ in the momentum equation (the second equation of~\eqref{Emass_momentum}) in order to obtain 
\begin{equation}\label{Etest1}
\begin{aligned}
\int_{\T^d}{\rho u}\cdot U(\tau)dx=&\int_{\T^d}\rho^0|u^0|^2dx+\int_0^\tau\int_{\T^d}{\rho u}\cdot\partial_t U+{(\rho u\otimes u)}:\nabla U dxdt\\
&+\int_0^\tau\int_{\T^d}{\rho^\gamma}\diverg U dxdt.
\end{aligned}
\end{equation}
Similarly, putting $\psi=\frac{1}{2}|U|^2$ and then $\psi=\gamma R^{\gamma-1}$ in the first equation of~\eqref{Emass_momentum} gives
\begin{equation}\label{Etest2}
\frac{1}{2}\int_{\T^d}|U(\tau)|^2{\rho}(\tau,x)dx=\int_0^\tau\int_{\T^d}U\cdot\partial_tU {\rho}+\nabla U U\cdot{\rho u}dxdt+\int_{\T^n}\frac{1}{2}|u^0|^2\rho^0dx
\end{equation}
and
\begin{equation}\label{Etest3}
\int_{\T^d}\gamma R^{\gamma-1}(\tau){\rho}(\tau)dx=\int_0^\tau\int_{\T^d}\gamma(\gamma-1)R^{\gamma-2}\partial_tR {\rho}+\gamma(\gamma-1)R^{\gamma-2}\nabla R \cdot{\rho u}dxdt+\int_{\T^d}\gamma (\rho^0)^\gamma dx,
\end{equation}
respectively. 

Next, write the relative energy as
\begin{equation*}
\begin{aligned}
E_{rel}(\tau)&=\int_{\T^d}\frac{1}{2}{\rho|u|^2}+\frac{1}{\gamma-1}{\rho^\gamma}dx + \int_{\T^d}R^\gamma dx+\frac{1}{2}\int_{\T^d}|U|^2{\rho}dx-\int_{\T^d}U\cdot {\rho u}dx-\int_{\T^d}\frac{\gamma}{\gamma-1}R^{\gamma-1}{\rho}dx\\
&=E(\tau)+ \int_{\T^d}R^\gamma dx+\frac{1}{2}\int_{\T^d}|U|^2{\rho}dx-\int_{\T^d}U\cdot{\rho u}dx-\int_{\T^d}\frac{\gamma}{\gamma-1}R^{\gamma-1}\rho dx
\end{aligned}
\end{equation*}
(all integrands evaluated at time $\tau$). Now, using the balances~\eqref{Etest1},~\eqref{Etest2},~\eqref{Etest3} for the last three integrals, we obtain
\begin{equation*}
\begin{aligned}
E_{rel}(\tau)=E(\tau)&+\int_{\T^d}R^\gamma dx\\
&+\int_0^\tau\int_{\T^d}U\cdot\partial_tU \rho+\nabla U U\cdot{\rho u}dxdt+\int_{\T^d}\frac{1}{2}|u^0|^2\rho^0dx\\ 
&-\int_{\T^d}\rho^0|u^0|^2dx-\int_0^\tau\int_{\T^d}{\rho u}\cdot\partial_t U+{(\rho u\otimes u)}:\nabla U dxdt\\
&-\int_0^\tau\int_{\T^d}{\rho^\gamma}\diverg U dxdt\\
&-\int_0^\tau\int_{\T^d}(\gamma R^{\gamma-2}\partial_tR \rho+\gamma R^{\gamma-2}\nabla R \cdot{\rho u})dxdt-\int_{\T^d}\frac{\gamma}{\gamma-1} (\rho^0)^\gamma dx,
\end{aligned}
\end{equation*}
and using~\eqref{Emvsenergy} we see, for a.e.\ $\tau$,
\begin{equation}\label{Eintermediatestep}
\begin{aligned}
E_{rel}(\tau)\leq& -\int_{\T^d}(\rho^0)^\gamma dx +\int_0^\tau\int_{\T^d}R^\gamma dx\\
&+\int_0^\tau\int_{\T^d}U\cdot\partial_tU \rho+\nabla U U\cdot{\rho u}dxdt\\ 
&-\int_0^\tau\int_{\T^d}{\rho u}\cdot\partial_t U+{(\rho u\otimes u)}:\nabla U dxdt\\
&-\int_0^\tau\int_{\T^d}{\rho^\gamma}\diverg U dxdt\\
&-\int_0^\tau\int_{\T^d}(\gamma R^{\gamma-2}\partial_tR \rho+\gamma R^{\gamma-2}\nabla R \cdot {\rho u})dxdt.
\end{aligned}
\end{equation}
Let us collect some terms and write
\begin{equation}\label{Eequality1}
\begin{aligned}
\int_{\T^d}&R^\gamma dx -\int_{\T^d}(\rho^0)^\gamma dx-\int_0^\tau\int_{\T^d}\gamma R^{\gamma-2}\partial_tR \rho dxdt\\
&=\int_0^\tau\int_{\T^d}\frac{d}{dt}R^\gamma-\gamma R^{\gamma-2}\partial_tR \rho dxdt\\
&=\int_0^\tau\int_{\T^d}\gamma R^{\gamma-1}\partial_tR-\gamma R^{\gamma-2}\partial_tR \rho dxdt\\
&=\int_0^\tau\int_{\T^d}\gamma R^{\gamma-2}\partial_tR(R-\rho)dxdt
\end{aligned}
\end{equation}
and
\begin{equation}\label{Eequality3}
\begin{aligned}
\int_0^\tau\int_{\T^d}&U\cdot\partial_tU \rho+\nabla U U\cdot{\rho u}-{\rho u}\cdot\partial_t U-{(\rho u\otimes u)}:\nabla U dxdt\\
&=\int_0^\tau\int_{\T^d}\partial_tU\cdot {\rho (U-u)}+\nabla U :{(\rho u\otimes(U-u))}dxdt.
\end{aligned}
\end{equation}

Insert equalities~\eqref{Eequality1} and~\eqref{Eequality3} into~\eqref{Eintermediatestep} to arrive at
\begin{equation}\label{Eintermediatestep2}
\begin{aligned}
E_{rel}(\tau)\leq&\int_0^\tau\int_{\T^d}\gamma R^{\gamma-2}\partial_tR(R-\rho)dxdt \\
&+\int_0^\tau\int_{\T^d}\partial_tU\cdot {\rho(U-u)}+\nabla U :{(\rho u\otimes(U-u))}dxdt\\
&-\int_0^\tau\int_{\T^d}{\rho^\gamma}\diverg U dxdt-\int_0^\tau\int_{\T^d}\gamma R^{\gamma-2}\nabla R \cdot{\rho u}dxdt.
\end{aligned}
\end{equation}

For the last two integrals, we have, using the divergence theorem,
\begin{equation*}
\begin{aligned}
-\int_0^\tau\int_{\T^d}&{\rho^\gamma}\diverg U dxdt-\int_0^\tau\int_{\T^d}\gamma R^{\gamma-2}\nabla R \cdot{\rho u}dxdt\\
&=\int_0^\tau\int_{\T^d}-{\rho^\gamma}\diverg U +\gamma R^{\gamma-2}\nabla R \cdot(RU-{\rho u})-\gamma R^{\gamma-2}\nabla R\cdot RUdxdt\\
&=\int_0^\tau\int_{\T^d}(R^\gamma-{\rho^\gamma})\diverg U +\gamma R^{\gamma-2}\nabla R \cdot(RU-{\rho u})dxdt.\\
\end{aligned}
\end{equation*}
Plugging this back into~\eqref{Eintermediatestep2} and observing that, by the mass equation for $(R,U)$, 
\begin{equation*}
\begin{aligned}
\gamma &R^{\gamma-2}\partial_tR(R-\rho)+\gamma R^{\gamma-2}\diverg UR(R-\rho) +\gamma R^{\gamma-2}\nabla R \cdot RU
=\gamma R^{\gamma-2}U\cdot\nabla R\rho,
\end{aligned}
\end{equation*}
we obtain
\begin{equation}\label{Eintermediatestep3}
\begin{aligned}
E_{rel}(\tau)\leq&\int_0^\tau\int_{\T^d}\gamma R^{\gamma-2}\cdot\nabla R{\rho(U-u)}dxdt \\
&+\int_0^\tau\int_{\T^d}\partial_tU\cdot {\rho(U-u)}+\nabla U :{(\rho u\otimes(U-u))}dxdt\\
&-\int_0^\tau\int_{\T^d}\gamma R^{\gamma-1}\diverg U(R-\rho )dxdt+\int_0^\tau\int_{\T^d}(R^\gamma-{\rho^\gamma})\diverg U.
\end{aligned}
\end{equation}
The expression in the third line is rewritten as 
\begin{equation}\label{Epointwiseest}
\begin{aligned}
\partial_tU&\cdot {\rho(U-u)}+\nabla U :{(\rho u\otimes(U-u))}\\
&=\partial_tU\cdot {\rho(U-u)}+\nabla U :{(\rho U\otimes(U-u))}+\nabla U :{(\rho(u-U)\otimes(U-u))},
\end{aligned}
\end{equation}
and the integral of the last expression as well as the last line in~\eqref{Eintermediatestep3} can both be estimated by
\begin{equation}\label{Eabsorb}
C\norm{U}_{C^1}\int_0^\tau E_{rel}(t)dt.
\end{equation}
For the other terms in~\eqref{Epointwiseest} we get, invoking the momentum equation for $(R,U)$,
\begin{equation}\label{Emomentum}
\begin{aligned}
\partial_tU&\cdot {\rho(U-u)}+\nabla U :U\otimes{\rho(U-u)}\\
&=\frac{1}{R}(\partial_t(RU)+\diverg(RU\otimes U))\cdot{\rho(U-u)}\\
&= - \gamma R^{\gamma-2}\nabla R\cdot{\rho(U-u)}.
\end{aligned}
\end{equation}
Putting together~\eqref{Eintermediatestep3},~\eqref{Eabsorb}, and~\eqref{Emomentum}, we obtain
\begin{equation*}
E_{rel}(\tau)\leq C\norm{U}_{C^1}\int_0^\tau E_{rel}(t)dt.
\end{equation*}
From Gr\"onwall's inequality it then follows that $E_{rel}(\tau)=0$ for a.e. $t$.
\end{proof}

Let us summarise the strategy we used in both examples:
\begin{enumerate}
\item Define a \term{relative energy} between a weak and a strong solution that is non-negative and that vanishes identically if and only if the two solutions coincide.
\item Using the weak formulation of the equations with suitable expressions of the strong solution as test functions, rewrite the relative energy at time $\tau$ in terms of the initial data and a time integral from $0$ to $\tau$.
\item Cancel the initial terms by virtue of the energy equality for the strong solution and the admissibility criterion for the weak solution.
\item Using the pointwise form of the equations and the $C^1$ bounds for the strong solution, ``absorb" the terms in the time integral into the relative energy.
\item Apply Gr\"onwall's inequality to conclude that the relative energy vanishes identically.       
\end{enumerate}

\section{Dissipative Measure-Valued Solutions}\label{mvs}
The problem with the results from the previous section is that they state weak-strong uniqueness only within a class for which existence is not known. Indeed, the existence of admissible weak solutions of the Euler equations for any initial data is still open (both in the compressible and in the incompressible case). For the application to singular limits as indicated in the introduction, it is therefore desirable to extend the notion of weak solution. We treat here only the incompressible system.

A conceivable approach to obtaining solutions to the Euler equations is to consider the \term{Navier-Stokes equations},
\begin{equation*}
\begin{aligned}
\partial_t u + \diverg(u\otimes u)+\nabla p&=\nu\Delta u,\\
\diverg u&=0,
\end{aligned}
\end{equation*} 
for which the existence of weak solutions is classically known, and to let the viscosity $\nu\searrow 0$. The na\"ive hope is that the corresponding solutions will converge (at least weakly in the sense of distributions) to a weak solution of the Euler equations. However, this convergence might be obstructed by the appearance of \term{oscillations} and \term{concentrations}. For instance, the sequence $u_n(x)=\sin(nx)$ converges weakly to zero, but the sequence of the squares $u_n^2$ does \emph{not} converge to the square of the limit (which is zero). In other words, weak convergence and nonlinearities do not, in general, commute. Let us mention that it is not known whether oscillations or concentrations can actually form in the viscosity sequence. The question is considered extremely difficult.

One way to handle such issues is to consider \term{measure-valued solutions}. We present here the formalism of \cite{feireislgwiazdawsu}, which in turn is inspired by \cite{DST}. Our framework is slightly simpler and more general than the ones commonly used \cite{dipernamajda, alibert}, and arguably represents the most general concept of solution still having the weak-strong uniqueness property (cf.\ the discussion in \cite{feireislgwiazdawsu}). Another advantage is that this framework allows to deal with viscosity terms.

Let $u_\nu$ denote a weak (Leray) solution\footnote{See the next section for a definition.} for the Navier-Stokes equations with viscosity $\nu$. If we identify the vector field $u_\nu$ with the family
\begin{equation*}
\{\delta_{u_\nu}(x,t)\}_{(x,t)\in\T^d\times(0,T)}
\end{equation*}
of (Dirac) probability measures, we can pass to the weak* limit in the space of parametrised probability measures along a subsequence. This limit is then compatible with nonlinearities. The facilitated convergence argument is thus ``traded" for the probabilistic relaxation: Indeed the limit measure may no longer be Dirac, so that the exact value of the velocity field is ignored and only its probability distribution is described.

To make this discussion more rigorous, we define a \emph{Young measure} to be a family 
\begin{equation*}
\{\nu_{x,t}\}_{(x,t)\in\T^d\times(0,T)}
\end{equation*}  
of probability measures on $\R^d$ that is \term{weakly* measurable}, i.e.\ the map
\begin{equation*}
(x,t)\mapsto \int_{\R^d}f(z)d\nu_{x,t}(z)
\end{equation*} 
is Borel measurable for any $f\in C_b(\R^d)$. Then the basic weak convergence statement (which follows basically from the Banach-Alaoglu Theorem) is as follows:

\begin{theorem}[Fundamental Theorem of Young Measures] 
Let $\{u_n\}_{n\in \N}$ be an $L^2$-bounded sequence of maps $\T^d\times (0,T)\to\R^d$. Then there is a subsequence (still denoted $\{u_n\}$) that generates a Young measure $\nu$, in the sense that 
\begin{equation*}
f\circ u_n\rightharpoonup \langle \nu,f\rangle:=\int_{\R^d}f(z)d\nu(z)
\end{equation*}
for any $f\in C(\R^d)$ for which $\{f\circ u_n\}_{n\in\N}$ is equiintegrable. Moreover,
\begin{equation}\label{measurebound}
\int_0^T\int_{\T^d}\langle\nu_{x,t},|\cdot|^2\rangle dxdt<\infty.
\end{equation}
\end{theorem}
Applying this to our vanishing viscosity sequence $\{u_n\}$ with associated viscosities $\nu(n)\searrow 0$, we see that the uniform $L^2$ bound on the $u_n$ implies in particular that $\{u_n\}$ itself is equiintegrable, so that $f(z)=z$ is an admissible test function in the Fundamental Theorem. Therefore, we can easily pass to the limit in the linear terms: Indeed, for any $\phi\in C^1(\T^d\times[0,T];\R^d)$,
\begin{equation*}
\int_0^T\int_{\T^d}u_n(x,t)\cdot\partial_t\phi(x,t)dxdt\to\int_0^T\int_{\T^d}\langle\nu_{x,t},\operatorname{id}\rangle\cdot\partial_t\phi(x,t)dxdt,
\end{equation*}  
\begin{equation*}
\nu(n)\int_0^T\int_{\T^d}u_n(x,t)\cdot\Delta\phi(x,t)dxdt\to0,
\end{equation*}  
and, for any $\psi\in C^1(\T^d\times[0,T])$,
\begin{equation*}
\int_0^T\int_{\T^d}u_n(x,t)\cdot\nabla\psi(x,t)dxdt\to\int_0^T\int_{\T^d}\langle\nu_{x,t},\operatorname{id}\rangle\cdot\nabla\psi(x,t)dxdt.
\end{equation*}  
The nonlinear term, however, may still lack compactness due to possible concentrations: For the sequence $\{u_n\otimes u_n\}$ we merely have an $L^\infty((0,T);L^1(\T^d))$ bound, which does not guarantee equiintegrability. Hence the difference
\begin{equation*}
u_n\otimes u_n-\langle\nu,\operatorname{id}\otimes\operatorname{id}\rangle
\end{equation*}  
may not converge to zero. It will, however, converge weakly* to a bounded measure $m$ on $\T^d\times[0,T]$ by virtue of the  $L^\infty((0,T);L^1(\T^d))$ bound and \eqref{measurebound}. By standard measure theory arguments, it is also not difficult to show that the $L^\infty((0,T);L^1(\T^d))$ bound implies the validity of the disintegration
\begin{equation}\label{disint}
m(dxdt)=m_t(dx)\otimes dt
\end{equation}
for some family $\{m_t\}_{t\in(0,T)}$ of uniformly bounded measures on $\R^d$. We thus obtain for any divergence-free $\phi\in C^1(\T^d\times[0,T];\R^d)$ and almost every $\tau\in(0,T)$ the equation

\begin{equation}\label{mvequation}
\begin{aligned}
\int_0^\tau\int_{\T^d}&\partial_t\phi(x,t)\cdot \langle\nu_{x,t},\operatorname{id}\rangle+\nabla\phi(x,t):(\langle\nu_{x,t},\operatorname{id}\otimes\operatorname{id}\rangle+m_t) dxdt\\
&=\int_{\T^d} \langle\nu_{x,\tau},\operatorname{id}\rangle\cdot\phi(x,\tau)-u^0(x)\cdot\phi(x,0)dx,
\end{aligned}
\end{equation}
provided the initial data $u^0$ was kept fixed along the viscosity sequence. Moreover, for any $\psi\in C^1(\T^d)$ we have the divergence condition
\begin{equation}\label{mvdiv}
\int_{\T^d} \langle\nu_{x,t},\operatorname{id}\rangle\cdot\nabla\psi(x)dx=0
\end{equation}
for almost every $t\in(0,T)$.

As a final ingredient in our definition of dissipative measure-valued solutions, we need to impose an analogue of the energy admissibility condition. To this end, define for almost every $\tau\in (0,T)$ the \term{dissipation defect}
\begin{equation*}
D(\tau)=\frac{1}{2}\trace m_\tau(\T^d)
\end{equation*}
and the \term{measure-valued energy} as
\begin{equation*}
E(\tau)=\frac{1}{2}\int_{\T^d}\langle \nu_{x,\tau},|\cdot|^2\rangle dx+D(\tau).
\end{equation*}
First, since $m_\tau$ is uniformly bounded (in the sense of measures) in $\tau$, we find $D\in L^\infty(0,T)$. Second, $D\geq0$, because for any non-negative $\chi\in C([0,T])$ and any $M>0$,
\begin{equation*}
\begin{aligned}
2\int_0^T\chi(\tau)D(\tau)d\tau&=\trace \int_0^T\int_{\T^d}\chi(\tau)dm(x,\tau)\\
&=\lim_{n\to\infty}\int_0^T\int_{T^d}\chi(\tau)\trace[u_n\otimes u_n-\langle\nu,\operatorname{id}\otimes\operatorname{id}\rangle]dxd\tau\\
&=\lim_{n\to\infty}\int_0^T\int_{T^d}\chi(\tau)(|u_n|^2-\langle\nu,|\cdot|^2\rangle dxd\tau\\
&=\lim_{n\to\infty}\left\{\int_0^T\int_{T^d}\chi(\tau)(|u_n|^2\wedge M-\langle\nu,|\cdot|^2\wedge M\rangle dxd\tau\right.\\
&\hspace{1.5cm}+\left.\int_0^T\int_{T^d}\chi(\tau)((|u_n|^2-M)\lor 0-\langle\nu,(|\cdot|^2-M)\lor 0\rangle dxd\tau\right\}\\
&\geq -\int_0^T\int_{T^d}\chi(\tau)\langle\nu,(|\cdot|^2-M)\lor 0\rangle dxd\tau.
\end{aligned}
\end{equation*}
For the last inequality, we used the Fundamental Theorem for the test function $|\cdot|^2\land M\in C_b(\R^d)$. But the remaining term becomes arbitrarily close to zero for sufficiently large $M$ by the dominated convergence theorem (since $(|\cdot|^2-M)\lor 0\leq|\cdot|^2$), whence the claim $D\geq0$ follows.

Third, for our viscosity sequence $\{u_n\}$ we have the energy bound
\begin{equation*}
\frac{1}{2}\int_{\T^d}|u_n(x,\tau)|^2dx\leq \frac{1}{2}\int_{\T^d}|u^0|^2dx
\end{equation*} 
for almost every $\tau\in(0,T)$, and passing to the measure-valued limit thus yields
\begin{equation}\label{mvdiss}
E(\tau)\leq \frac{1}{2}\int_{\T^d}|u^0|^2dx
\end{equation}
for almost every $\tau\in(0,T)$.

We condense our discussion into the following definition:
\begin{definition}
Let $\nu$ be a Young measure, $m$ a matrix-valued measure on $\T^d\times[0,T]$ satisfying \eqref{disint}, and $D\in L^\infty(0,T)$ with $D\geq0$ such that 
\begin{equation*}
|m_t|(\T^d)\leq C D(t)
\end{equation*}
for some constant $C$ and almost every $t\in[0,T]$.

The triple $(\nu, m, D)$ is called a \term{dissipative measure-valued solution} of the Euler equations with initial datum $u^0$ if it satisfies \eqref{mvequation}, \eqref{mvdiv}, and \eqref{mvdiss}. 
\end{definition}
\begin{remark}
Note carefully that our definition does not require the measure-valued solution to arise from a viscosity limit. Indeed, if $u^0$ is kept fixed along the viscosity limit, then there exist dissipative measure-valued solutions (in fact even admissible weak solutions) which can not be recovered as a vanishing viscosity limit. 
\end{remark}

From the previous discussion it is clear that there exists a dissipative measure-valued solution for any $u^0\in L^2(\T^d)$ (just take a viscosity sequence). But we also have weak-strong uniqueness:
\begin{theorem}
Let $(\nu,m,D)$ be a dissipative measure-valued solution and $U\in C^1(\mathbb{T}^d\times[0,T])$ a strong solution of \eqref{euler}, both with initial datum $u^0$. Then $\nu_{x,\tau} =\delta_{U(x,\tau)}$ for almost every $(x,\tau)\in\To^d\times(0,T)$, and $m=0$, $D=0$.
\end{theorem}
\begin{proof}
The proof is achieved essentially by rewriting the proof of Theorem \ref{basicthm} in measure-valued notation. Specifically, we estimate for almost every $\tau\in(0,T)$ the \term{relative energy} $E_{rel}$ as 
\begin{align*}
E_{rel}(\tau)&:= \frac{1}{2}\int_{\To^d}\langle \nu_{x,\tau},|\operatorname{id}-U(x,\tau)|^2\rangle dx+D(\tau)\\
&=\frac{1}{2}\int_{\To^3}|U(x,\tau)|^2dx+\frac{1}{2}\int_{\To^d}\langle\nu_{x,\tau},|\cdot|^2\rangle dx+D(\tau)-\int_{\To^d}\langle \nu_{x,\tau},\operatorname{id}\rangle\cdot U(x,\tau)dx\\
&\leq \frac{1}{2}\int_{\To^3}|u^0(x)|^2dx+\frac{1}{2}\int_{\To^d}|u^0(x)|^2dx-\int_{\To^d}\langle \nu_{x,\tau},\operatorname{id}\rangle\cdot U(x,\tau)dx\\
&=\int_{\To^3}|u^0(x)|^2dx-\int_{\To^3}u^0(x)\cdot u^0(x)dx\\
&\hspace{1cm}-\int_0^\tau\int_{\To^d}[\partial_tU(x,t)\cdot\langle \nu_{x,t},\operatorname{id}\rangle+\nabla U(x,t):\langle \nu_{x,t},\operatorname{id}\otimes\operatorname{id}\rangle]dxdt\\
&\hspace{1cm}-\int_0^\tau\int_{\To^d}\nabla U(x,t)dm(x,t)\\
&=-\int_0^\tau\int_{\To^d}[\partial_tU(x,t)\cdot\langle \nu_{x,t},\operatorname{id}\rangle+\nabla U(x,t):\langle \nu_{x,t},\operatorname{id}\otimes\operatorname{id}\rangle]dxdt\\
&\hspace{1cm}-\int_0^\tau\int_{\To^d}\nabla_{sym} U(x,t)dm(x,t)\\
&=\int_0^\tau\int_{\To^d}[\diverg(U\otimes U)(x,t)\cdot\langle \nu_{x,t},\operatorname{id}\rangle-\nabla U(x,t)\langle \nu_{x,t},\operatorname{id}\otimes\operatorname{id}\rangle]dxdt\\
&\hspace{1cm}-\int_0^\tau\int_{\To^d}\nabla_{sym} U(x,t)dm(x,t)\\
&=\int_0^\tau\int_{\To^d}\langle\nu_{x,t},(U(x,t)-\operatorname{id})\cdot\nabla_{sym} U(x,t)(\operatorname{id}-U(x,t))\rangle dxdt\\
&\hspace{1cm}-\int_0^\tau\int_{\To^d}\nabla_{sym} U(x,t)dm(x,t)\\
&\leq \int_0^\tau\norm{\nabla_{sym}U(t)}_{L^\infty(\To^d)}\left[\int_{\To^d}\langle\nu_{x,t},|\operatorname{id}-U(x,t)|^2\rangle dx+|m_t|(\T^d)\right]dt\\
&\leq C\int_0^\tau\norm{\nabla_{sym}U(t)}_{L^\infty(\To^d)}E_{rel}(t)dt,
\end{align*}
and from Gr\"onwall's inequality it follows that $E_{rel}(\tau)=0$ almost everywhere. But since both terms in the relative energy are non-negative, this implies indeed $\nu_{x,\tau}=\delta_{U(x,\tau)}$ and $D(\tau)=0$ almost everywhere, and $m=0$ follows from $|m_\tau|(\T^d)\leq CD(\tau)$ for almost every $\tau$. 
\end{proof}

\section{Dealing with Viscosity}\label{viscosity}
\subsection{Incompressible Navier-Stokes Equations}
For the incompressible Navier-Stokes equations
\begin{equation}\label{navierstokes}
\begin{aligned}
\partial_t u + \diverg(u\otimes u)+\nabla p&=\nu\Delta u,\\
\diverg u&=0,
\end{aligned}
\end{equation}
the improved regularity thanks to the diffusion term enables us to show a much better weak-strong uniqueness result than the one for Euler (Theorem \ref{basicthm}). ``Better" means that the ``strong" solution is allowed to be much less regular that Lipschitz. In our discussion of the incompressible Navier-Stokes equations, we always assume $d=2$ or $d=3$, as Sobolev embeddings come into play which are no longer valid in higher dimensions. 

For the Navier-Stokes equations, we can afford to work in the more realistic setting of bounded domains. Indeed, the no-slip boundary conditions
\begin{equation*}
u\cdot n=0\hspace{0.3cm}\text{on $\partial\Omega$}
\end{equation*}
that are most commonly imposed on the equations do not conflict with weak-strong uniqueness. The situation for the Euler equations is vastly different, see Section \ref{boundary} below. In what follows, let $\Omega\subset \R^d$ be a bounded Lipschitz domain. 

Before we state the result, let us discuss some preliminaries:
\begin{definition}[\cite{leray}]
A vectorfield $u\in L^\infty((0,T);L^2(\Omega))\cap L^2((0,T);H_0^1(\Omega))$ is called a \term{Leray solution} of \eqref{navierstokes} with inital datum $u^0$ if 
\begin{equation*}
\begin{aligned}
&\int_0^\tau\int_{\Omega}\partial_t\phi(x,t)\cdot u(x,t)-\phi(x,t)\cdot (u(x,t)\cdot\nabla u(x,t))-\nu\nabla\phi(x,t):\nabla u(x,t)dxdt\\
&\hspace{3cm}=\int_{\Omega}u(x,\tau)\cdot\phi(x,\tau)-u^0(x)\cdot\phi(x,0)dx
\end{aligned}
\end{equation*}
for almost every $\tau\in(0,T)$ and every $\phi\in C_c^1(\Omega\times[0,T];\R^d)$ with $\diverg\phi=0$,
\begin{equation*}
\int_{\Omega}u(x,\tau)\cdot\nabla \psi(x)dx=0
\end{equation*}
for almost every $\tau\in(0,T)$ and every $\psi\in C^1(\Omega)$, and
\begin{equation*}
\frac{1}{2}\int_{\Omega}|u(x,\tau)|^2dx+\nu\int_0^\tau\int_{\Omega}|\nabla u(x,t)|^2dxdt\leq\frac{1}{2}\int_{\Omega}|u^0(x)|^2dx
\end{equation*}
for almost every $\tau\in(0,T)$.
\end{definition} 
We will use the following inequality, which can easily be verified via H\"older and Sobolev inequalities: Let $r,s$ be such that
\begin{equation*}
\frac{d}{s}+\frac{2}{r}=1,\hspace{0.4cm}d< s<\infty,
\end{equation*} 
 and let $u,U\in  L^\infty((0,T);L^2(\Omega))\cap L^2((0,T);H_0^1(\Omega))$ be divergence-free. Assume in addtion $U\in L^r((0,T);L^s(\Omega))$.
Then, for every $\tau\in(0,T)$,
\begin{equation*}
\begin{aligned}
&\left|\int_0^\tau\int_{\Omega}\nabla(u-U):((u-U)\otimes U)dxdt\right|\\
&\hspace{0.5cm}\leq C\left(\int_0^\tau\int_{\Omega}|\nabla(u-U)|^2dxdt\right)^{1-1/r}\left(\int_0^\tau\norm{U}_{L_x^s}^r\int_{\Omega}|u-U|^2dxdt\right)^{1/r}.
\end{aligned}
\end{equation*}

The classical weak-strong uniqueness theorem of Prodi \cite{prodi} and Serrin \cite{serrin} reads:
\begin{theorem}
Let $d\in\{2,3\}$, $r,s$ as above and $u$, $U$ be two Leray solutions with the same initial datum $u^0$. If $U\in L^r((0,T);L^s(\T^d))$, then $u=U$ for almost every $(x,\tau)\in\T^d\times(0,T)$. 
\end{theorem} 
\begin{proof}Assume for the sake of this proof sketch that $u$ and $v$ are smooth; the general case can then be handled by (nontrivial) approximation arguments detailed e.g.\ in Chapter 4 of \cite{galdinotes}.

We estimate for almost every $\tau\in(0,T)$
\begin{equation*}
\begin{aligned}
&\frac{1}{2}\int_{\Omega}|u-U|(\tau)^2dx+\nu\int_0^\tau\int_{\Omega}|\nabla(u-U)|^2dxdt\\
&= \frac{1}{2}\int_{\Omega}|u(\tau)|^2dx+\frac{1}{2}\int_{\Omega}|U(\tau)|^2dx+\nu\int_0^\tau\int_{\Omega}|\nabla u|^2dxdt+\nu\int_0^\tau\int_{\Omega}|\nabla U|^2dxdt\\
&\hspace{1cm}-\int_{\Omega}u\cdot U(\tau)dx-2\nu\int_0^\tau\int_{\Omega}\nabla u:\nabla Udxdt\\
&\leq \int_{\Omega}|u^0|^2dx-\int_{\Omega}u\cdot U(\tau)dx-2\nu\int_0^\tau\int_{\Omega}\nabla u:\nabla Udxdt\\
&=\int_0^\tau\int_{\Omega}-\partial_tU\cdot u+U\cdot(u\cdot\nabla)u+\nu\nabla U:\nabla u dxdt-2\nu\int_0^\tau\int_{\Omega}\nabla u:\nabla Udxdt\\
&=\int_0^\tau\int_{\Omega}-\partial_tU\cdot u+U\cdot(u\cdot\nabla)u-\nu\nabla U:\nabla u dxdt\\
&=\int_0^\tau\int_{\Omega}U\cdot\partial_t u+U\cdot(u\cdot\nabla)u-\nu\nabla U:\nabla u dxdt+\int_{\Omega}|u^0|^2dx-\int_{\Omega}u\cdot U(\tau)dx\\
&=\int_0^\tau\int_{\Omega}-\nabla u:(U\otimes U)+U\cdot(u\cdot\nabla)udxdt\\
&=\int_0^\tau\int_{\Omega}\nabla u:((u-U)\otimes U)dxdt\\
&=\int_0^\tau\int_{\Omega}\nabla (u-U):((u-U)\otimes U)dxdt\\
&\leq C\left(\int_0^\tau\int_{\Omega}|\nabla(u-U)|^2dxdt\right)^{1-1/r}\left(\int_0^\tau\norm{U}_{L_x^s}^r\int_{\Omega}|u-U|^2dxdt\right)^{1/r}\\
&\leq \nu \int_0^\tau\int_{\Omega}|\nabla(u-U)|^2dxdt+C(\nu)\int_0^\tau\norm{U}_{L_x^s}^r\int_{\Omega}|u-U|^2dxdt,
\end{aligned}
\end{equation*}
where in the last step we used Young's inequality 
\begin{equation*}
ab\leq\frac{\delta a^p}{p}+\frac{b^q}{\delta^{q/p} q}
\end{equation*}
for any $\delta>0$, $1/p+1/q=1$, and $a,b\geq0$. 

Since the terms with factor $\nu$ vanish on both sides, we can use Gr\"onwall's inequality to conclude.
\end{proof}

\begin{remark}
The statement is also true in the endpoint cases $s\in\{d,\infty\}$, with slightly modified proofs. See again \cite{galdinotes}, Chapter 4, for details.
\end{remark}

It is worthwhile noting the fundamental differences between this proof and the one for the incompressible Euler equations: For Navier-Stokes, we use both $u$ and $U$ as the test field in the definition of weak solution of the respective other function (cf.\ lines 3 and 7 in the above estimate); and, most importantly, the compensation of the (insufficient) exponent $1/r$ by the viscosity term in the last step relies crucially on $\nu>0$ (indeed, the constant $C(\nu)$ blows up as $\nu\searrow 0$).

\subsection{Compressible Navier-Stokes Equations}
In the case of the incompressible Navier-Stokes equations, the Laplacian helped us achieve a much better weak-strong uniqueness result than for the inviscid (Euler) system. The situation is much different for the \term{isentropic compressible Navier-Stokes equations}, at least as far as we currently know. Here, the viscosity term presents an additional difficulty that needs to be overcome to show weak-strong uniqueness at the same regularity as for the compressible Euler system. It is an interesting open question whether the result can be improved by a more clever treatment of the viscous term.

The equations read 
\begin{equation}\label{isentropicNS}
\begin{aligned}
\partial_t (\rho u) + \diverg(\rho u\otimes u)+\nabla \rho^\gamma&=\diverg\mathbb{S}(\nabla u),\\
\partial_t\rho+\diverg (\rho u)&=0,
\end{aligned}
\end{equation}
where 
\begin{equation*}
\mathbb{S}(\nabla u):=\mu\left(\nabla u+\nabla^t u-\frac{2}{3}(\diverg u) I\right)+\eta(\diverg u) I
\end{equation*}
is the \term{Newtonian viscous stress} for some given parameters $\mu>0$ and $\eta\geq0$.  

Similarly as before, a pair $(\rho,u)$ is a \term{weak solution} of the isentropic Navier-Stokes system with initial datum $(\rho^0,u^0)$ if $\rho\in L^\gamma(\T^d\times(0,T))$, $\rho|u|^2\in L^1(\T^d\times (0,T))$, $u\in L^2(0,T;H_0^1(\T^d))$, and

\begin{equation}\label{NSmass_momentum}
\begin{aligned}
\int_0^\tau\int_{\Omega}\partial_t\psi \rho+\nabla\psi\cdot{\rho u}dxdt+\int_{\Omega}\psi(x,0)\rho^0-\psi(x,\tau)\rho(x,\tau)dx&=0,\\
\int_0^\tau\int_{\Omega}\partial_t\phi\cdot{\rho u}+\nabla\phi : {(\rho u\otimes u)}+\diverg\phi{\rho^\gamma}-\mathbb{S}(\nabla u):\nabla \phi dxdt\\
+\int_{\Omega}\phi(x,0)\cdot \rho^0u^0-\phi(x,\tau)\cdot{\rho u}(x,\tau)dx&=0
\end{aligned}
\end{equation}
for almost every $\tau\in(0,T)$, every $\psi\in C^1(\Omega\times[0,T])$, and every $\phi\in C_c^1(\Omega\times[0,T];\R^d)$. Again, we can impose a weak energy inequality: A solution is \term{globally admissible} if for almost every $\tau\in(0,T)$
\begin{equation}\label{NSmvsenergy}
E(\tau)+\int_0^\tau\int_{\Omega}\mathbb{S}(\nabla u):\nabla u dxdt\leq E^0,
\end{equation}
where the energy $E$ is defined as in the inviscid situation, i.e.
\begin{equation*}
E(\tau)=\frac{1}{2}\int_{\Omega}\rho(x,\tau)|u(x,\tau)|^2dx+\frac{1}{\gamma-1}\rho(x,\tau)^\gamma dx,
\end{equation*}
and
\begin{equation*}
E^0=\frac{1}{2}\int_{\Omega}\rho^0(x)|u^0(x)|^2dx+\frac{1}{\gamma-1}\rho^0(x)^\gamma dx.
\end{equation*}

\begin{theorem}
Let $(\rho,u)$ be a globally admissible weak solution of \eqref{isentropicNS} and $(R,U)$ a strong solution such that
\begin{equation*}
R\in C^1(\Omega\times[0,T]),\hspace{0.3cm}R>0,\hspace{0.3cm}U\in C^1([0,T];C^2(\Omega)),\hspace{0.3cm}U\cdot n=0\hspace{0.2cm}\text{on $\partial\Omega$}.
\end{equation*}
If both solutions share the same initial data, then they coincide for almost every $(x,t)\in\Omega\times(0,T)$.
\end{theorem}
\begin{proof}
Large parts of the proof are very similar to the case of the isentropic Euler equations, see Section \ref{relenergy}. We give the full argument nevertheless for the reader's convenience.

Agein we define for a.e.\ $\tau\in[0,T]$ the relative energy between $(R,U)$ and $(\rho,u)$ as
\begin{equation*}
E_{rel}(\tau)=\int_{\Omega}\frac{1}{2}{\rho|u-U|^2}+{\frac{1}{\gamma-1}\rho^\gamma-\frac{\gamma}{\gamma-1}R^{\gamma-1}\rho+R^\gamma}dx.
\end{equation*}
Note once more that, since $\gamma>1$, it holds that $E_{rel}(\tau)=0$ for a.e.\ $\tau$ implies Theorem~\ref{Eweak-strong}. 

Setting $\phi=U$ in the momentum equation gives 
\begin{equation}\label{NStest1}
\begin{aligned}
\int_{\Omega}{\rho u}\cdot U(\tau)dx=&\int_{\Omega}\rho^0|u^0|^2dx+\int_0^\tau\int_{\Omega}{\rho u}\cdot\partial_t U+{(\rho u\otimes u)}:\nabla U dxdt\\
&+\int_0^\tau\int_{\Omega}{\rho^\gamma}\diverg U dxdt-\int_0^\tau\int_{\Omega}\mathbb{S}(\nabla u):\nabla Udxdt.
\end{aligned}
\end{equation}
Then, setting $\psi=\frac{1}{2}|U|^2$ and $\psi=\gamma R^{\gamma-1}$ in the first equation of~\eqref{Emass_momentum} yields
\begin{equation}\label{NStest2}
\frac{1}{2}\int_{\Omega}|U(\tau)|^2{\rho}(\tau,x)dx=\int_0^\tau\int_{\Omega}U\cdot\partial_tU {\rho}+\nabla U U\cdot{\rho u}dxdt+\int_{\Omega}\frac{1}{2}|u^0|^2\rho^0dx
\end{equation}
and
\begin{equation}\label{NStest3}
\int_{\Omega}\gamma R^{\gamma-1}(\tau){\rho}(\tau)dx=\int_0^\tau\int_{\Omega}\gamma(\gamma-1)R^{\gamma-2}\partial_tR {\rho}+\gamma(\gamma-1)R^{\gamma-2}\nabla R \cdot{\rho u}dxdt+\int_{\Omega}\gamma (\rho^0)^\gamma dx,
\end{equation}
respectively.

We write the relative energy as
\begin{equation*}
\begin{aligned}
E_{rel}(\tau)&=\int_{\Omega}\frac{1}{2}{\rho|u|^2}+\frac{1}{\gamma-1}{\rho^\gamma}dx + \int_{\Omega}R^\gamma dx+\frac{1}{2}\int_{\Omega}|U|^2{\rho}dx-\int_{\Omega}U\cdot {\rho u}dx-\int_{\Omega}\frac{\gamma}{\gamma-1}R^{\gamma-1}{\rho}dx\\
&=E(\tau)+ \int_{\Omega}R^\gamma dx+\frac{1}{2}\int_{\Omega}|U|^2{\rho}dx-\int_{\Omega}U\cdot{\rho u}dx-\int_{\Omega}\frac{\gamma}{\gamma-1}R^{\gamma-1}\rho dx
\end{aligned}
\end{equation*}
Next, using the balances~\eqref{NStest1},~\eqref{NStest2},~\eqref{NStest3} for the last three integrals, we have
\begin{equation*}
\begin{aligned}
E_{rel}(\tau)=E(\tau)&+\int_{\Omega}R^\gamma dx\\
&+\int_0^\tau\int_{\Omega}U\cdot\partial_tU \rho+\nabla U U\cdot{\rho u}dxdt+\int_{\Omega}\frac{1}{2}|u^0|^2\rho^0dx\\ 
&-\int_{\Omega}\rho^0|u^0|^2dx-\int_0^\tau\int_{\Omega}{\rho u}\cdot\partial_t U+{(\rho u\otimes u)}:\nabla U dxdt\\
&-\int_0^\tau\int_{\Omega}{\rho^\gamma}\diverg U dxdt+\int_0^\tau\int_{\Omega}\mathbb{S}(\nabla u):\nabla Udxdt\\
&-\int_0^\tau\int_{\Omega}(\gamma R^{\gamma-2}\partial_tR \rho+\gamma R^{\gamma-2}\nabla R \cdot{\rho u})dxdt-\int_{\Omega}\frac{\gamma}{\gamma-1} (\rho^0)^\gamma dx,
\end{aligned}
\end{equation*}
and using~\eqref{NSmvsenergy} we have, for a.e.\ $\tau$,
\begin{equation}\label{NSintermediatestep}
\begin{aligned}
E_{rel}(\tau)+&\int_0^\tau\int_{\Omega}\mathbb{S}(\nabla u):(\nabla u-\nabla U)dxdt\\
\leq& -\int_{\Omega}(\rho^0)^\gamma dx +\int_0^\tau\int_{\Omega}R^\gamma dx\\
&+\int_0^\tau\int_{\Omega}U\cdot\partial_tU \rho+\nabla U U\cdot{\rho u}dxdt\\ 
&-\int_0^\tau\int_{\Omega}{\rho u}\cdot\partial_t U+{(\rho u\otimes u)}:\nabla U dxdt\\
&-\int_0^\tau\int_{\Omega}{\rho^\gamma}\diverg U dxdt\\
&-\int_0^\tau\int_{\Omega}(\gamma R^{\gamma-2}\partial_tR \rho+\gamma R^{\gamma-2}\nabla R \cdot {\rho u})dxdt.
\end{aligned}
\end{equation}
Next, we collect some terms:
\begin{equation}\label{NSequality1}
\begin{aligned}
\int_{\Omega}&R^\gamma dx -\int_{\Omega}(\rho^0)^\gamma dx-\int_0^\tau\int_{\Omega}\gamma R^{\gamma-2}\partial_tR \rho dxdt\\
&=\int_0^\tau\int_{\Omega}\frac{d}{dt}R^\gamma-\gamma R^{\gamma-2}\partial_tR \rho dxdt\\
&=\int_0^\tau\int_{\Omega}\gamma R^{\gamma-1}\partial_tR-\gamma R^{\gamma-2}\partial_tR \rho dxdt\\
&=\int_0^\tau\int_{\Omega}\gamma R^{\gamma-2}\partial_tR(R-\rho)dxdt
\end{aligned}
\end{equation}
and
\begin{equation}\label{NSequality3}
\begin{aligned}
\int_0^\tau\int_{\Omega}&U\cdot\partial_tU \rho+\nabla U U\cdot{\rho u}-{\rho u}\cdot\partial_t U-{(\rho u\otimes u)}:\nabla U dxdt\\
&=\int_0^\tau\int_{\Omega}\partial_tU\cdot {\rho (U-u)}+\nabla U :{(\rho u\otimes(U-u))}dxdt.
\end{aligned}
\end{equation}

Plugging equalities~\eqref{NSequality1} and~\eqref{NSequality3} into~\eqref{NSintermediatestep}, we obtain
\begin{equation}\label{NSintermediatestep2}
\begin{aligned}
E_{rel}(\tau)+&\int_0^\tau\int_{\Omega}\mathbb{S}(\nabla u):(\nabla u-\nabla U)dxdt\\
\leq&\int_0^\tau\int_{\Omega}\gamma R^{\gamma-2}\partial_tR(R-\rho)dxdt \\
&+\int_0^\tau\int_{\Omega}\partial_tU\cdot {\rho(U-u)}+\nabla U :{(\rho u\otimes(U-u))}dxdt\\
&-\int_0^\tau\int_{\Omega}{\rho^\gamma}\diverg U dxdt-\int_0^\tau\int_{\Omega}\gamma R^{\gamma-2}\nabla R \cdot{\rho u}dxdt.
\end{aligned}
\end{equation}
For the last two integrals, we have 
\begin{equation*}
\begin{aligned}
-\int_0^\tau\int_{\Omega}&{\rho^\gamma}\diverg U dxdt-\int_0^\tau\int_{\Omega}\gamma R^{\gamma-2}\nabla R \cdot{\rho u}dxdt\\
&=\int_0^\tau\int_{\Omega}-{\rho^\gamma}\diverg U +\gamma R^{\gamma-2}\nabla R \cdot(RU-{\rho u})-\gamma R^{\gamma-2}\nabla R\cdot RUdxdt\\
&=\int_0^\tau\int_{\Omega}(R^\gamma-{\rho^\gamma})\diverg U +\gamma R^{\gamma-2}\nabla R \cdot(RU-{\rho u})dxdt.\\
\end{aligned}
\end{equation*}
Inserting this back into~\eqref{NSintermediatestep2} and observing that, owing to the mass equation for $(R,U)$, 
\begin{equation*}
\begin{aligned}
\gamma &R^{\gamma-2}\partial_tR(R-\rho)+\gamma R^{\gamma-2}\diverg UR(R-\rho) +\gamma R^{\gamma-2}\nabla R \cdot RU
=\gamma R^{\gamma-2}U\cdot\nabla R\rho,
\end{aligned}
\end{equation*}
we get
\begin{equation}\label{NSintermediatestep3}
\begin{aligned}
E_{rel}(\tau)+&\int_0^\tau\int_{\Omega}\mathbb{S}(\nabla u):(\nabla u-\nabla U)dxdt\\
\leq&\int_0^\tau\int_{\Omega}\gamma R^{\gamma-2}\cdot\nabla R{\rho(U-u)}dxdt \\
&+\int_0^\tau\int_{\Omega}\partial_tU\cdot {\rho(U-u)}+\nabla U :{(\rho u\otimes(U-u))}dxdt\\
&-\int_0^\tau\int_{\Omega}\gamma R^{\gamma-1}\diverg U(R-\rho )dxdt+\int_0^\tau\int_{\Omega}(R^\gamma-{\rho^\gamma})\diverg U.
\end{aligned}
\end{equation}
The integrand in the third line can be rewritten pointwise as 
\begin{equation}\label{NSpointwiseest}
\begin{aligned}
\partial_tU&\cdot {\rho(U-u)}+\nabla U :{(\rho u\otimes(U-u))}\\
&=\partial_tU\cdot {\rho(U-u)}+\nabla U :{(\rho U\otimes(U-u))}+\nabla U :{(\rho(u-U)\otimes(U-u))},
\end{aligned}
\end{equation}
and the integral of the last term and the one from the last line in~\eqref{NSintermediatestep3} can be estimated by
\begin{equation}\label{NSabsorb}
C\norm{U}_{C^1}\int_0^\tau E_{rel}(t)dt.
\end{equation}
For the remaining terms in~\eqref{NSpointwiseest} we obtain by the momentum equation for $(R,U)$
\begin{equation}\label{NSmomentum}
\begin{aligned}
\partial_tU&\cdot {\rho(U-u)}+\nabla U :U\otimes{\rho(U-u)}\\
&=\frac{1}{R}(\partial_t(RU)+\diverg(RU\otimes U))\cdot{\rho(U-u)}\\
&= - \gamma R^{\gamma-2}\nabla R\cdot{\rho(U-u)}+\frac{1}{R}\diverg\mathbb{S}(\nabla U)\cdot\rho(U-u).
\end{aligned}
\end{equation}
Combining~\eqref{NSintermediatestep3},~\eqref{NSabsorb}, and~\eqref{NSmomentum}, we obtain
\begin{equation*}
\begin{aligned}
E_{rel}(\tau)&+\int_0^\tau\int_{\Omega}\mathbb{S}(\nabla u):(\nabla u-\nabla U)dxdt\\
\leq &C\norm{U}_{C^1}\int_0^\tau E_{rel}(t)dt+\int_0^\tau\int_{\Omega}\frac{1}{R}\diverg\mathbb{S}(\nabla U)\cdot\rho(U-u)dxdt
\end{aligned}
\end{equation*}
and thus
\begin{equation*}
\begin{aligned}
E_{rel}(\tau)&+\int_0^\tau\int_{\Omega}(\mathbb{S}(\nabla u)-\mathbb{S}(\nabla U)):(\nabla u-\nabla U)dxdt\\
\leq &C\norm{U}_{C^1}\int_0^\tau E_{rel}(t)dt+\int_0^\tau\int_{\Omega}\frac{1}{R}\diverg\mathbb{S}(\nabla U)\cdot\rho(U-u)-\mathbb{S}(\nabla U):(\nabla u-\nabla U)dxdt.
\end{aligned}
\end{equation*}
The desired result will follow from Gr\"onwall's inequality if we can absorb the remaining viscosity term
\begin{equation*}
\int_0^\tau\int_{\Omega}\frac{1}{R}\diverg\mathbb{S}(\nabla U)\cdot\rho(U-u)-\mathbb{S}(\nabla U):(\nabla u-\nabla U)dxdt
\end{equation*}
into the relative energy. First observe that, by the divergence theorem, 
\begin{equation*}
\begin{aligned}
&\int_0^\tau\int_{\Omega}\frac{1}{R}\diverg\mathbb{S}(\nabla U)\cdot\rho(U-u)-\mathbb{S}(\nabla U):(\nabla u-\nabla U)dxdt\\
&=\int_0^\tau\int_{\Omega}\frac{1}{R}\diverg\mathbb{S}(\nabla U)\cdot(\rho-R)(U-u)dxdt.
\end{aligned}
\end{equation*}
By our assumptions on $U$ and $R$, the factor $R^{-1}\diverg\mathbb{S}(\nabla U)$ is bounded in $L^\infty$, so it remains to estimate $\int\int|\rho-R||U-u|dxdt$.

To this end, we split the domain of integration into the three parts 
\begin{equation*}
\Omega_<\cup\Omega_\sim\cup\Omega_>:=\{\rho\leq  \underline{R}/2\}\cup\{\underline{R}/2<\rho<2\overline{R}\}\cup\{\rho\geq2\overline{R}\},
\end{equation*}
where $\underline{R}:=\inf_{\bar{\Omega}\times[0,T]}R(x,t)>0$ and $\overline{R}:=\sup_{\bar{\Omega}\times[0,T]}R(x,t)<\infty$.

Since the map $s\mapsto s^\gamma$ is smooth and strictly convex in $(0,\infty)$ and the range of $R$ is compact, there exists a constant (depending only on $R$, $\Omega$, $d$, and $\gamma$) such that on $\Omega_<\cup\Omega_\sim$,
\begin{equation*}
(\rho-R)^2\leq C \left({\frac{1}{\gamma-1}\rho^\gamma-\frac{\gamma}{\gamma-1}R^{\gamma-1}\rho+R^\gamma}\right).
\end{equation*}
On the other hand, on $\Omega_>$ the term $\rho^\gamma$ dominates, and therefore there exists another constant (also denoted $C$) such that on $\Omega_>$,
\begin{equation*}
\rho\leq  C \left({\frac{1}{\gamma-1}\rho^\gamma-\frac{\gamma}{\gamma-1}R^{\gamma-1}\rho+R^\gamma}\right).
\end{equation*}
Armed with these estimates, we find
\begin{equation*}
\begin{aligned}
\int\int_{\Omega_<}&|\rho-R||U-u|dxdt\\
&\leq C(\delta) \int_0^\tau\int_{\Omega}(\rho-R)^2dxdt+ \delta\int_0^\tau\int_{\Omega}|U-u|^2dxdt\\
&\leq C(\delta)\int_0^\tau E_{rel}(t)dt+\int_0^\tau\int_{\Omega}(\mathbb{S}(\nabla u)-\mathbb{S}(\nabla U)):(\nabla u-\nabla U)dxdt,
\end{aligned}
\end{equation*} 
where we chose a sufficiently small number $\delta>0$ and invoked the Poincar\'e-Korn inequality
\begin{equation*}
\int_\Omega|w|^2\leq C\int_{\Omega} \mathbb{S}(\nabla w):\nabla wdx
\end{equation*}
for all $w\in H_0^1(\Omega)$.

Further, we have
\begin{equation*}
\begin{aligned}
\int\int_{\Omega_\sim}&|\rho-R||U-u|dxdt\\
&\leq \frac{1}{2}\int\int_{\Omega_\sim}\frac{(\rho-R)^2}{\sqrt{\rho}}dxdt+ \frac{1}{2}\int\int_{\Omega_\sim}\frac{1}{\sqrt{\rho}}\rho|U-u|^2dxdt\\
&\leq C\int_0^\tau E_{rel}(t)dt,
\end{aligned}
\end{equation*}
since $\rho$ is bounded away from zero on $\Omega_\sim$.
\end{proof}
Finally, we estimate (keeping in mind $\rho>R$ on $\Omega_>$)
\begin{equation*}
\begin{aligned}
\int\int_{\Omega_>}&|\rho-R||U-u|dxdt\\
&\leq \int\int_{\Omega_>}\rho|U-u|dxdt\\
&\leq\frac{1}{2}\int\int_{\Omega_>}\rho dxdt+ \frac{1}{2}\int\int_{\Omega_>}\rho|U-u|^2dxdt\\
&\leq C \int_0^\tau E_{rel}(t)dt.
\end{aligned}
\end{equation*}
Putting everything together, we obtain $E_{rel}(\tau)\leq C\int_0^\tau E_{rel}(t)dt$, and the claim follows.

\section{Physical Boundaries in the Inviscid Situation}\label{boundary}
Choosing $\T^d$ as the spatial domain (thus imposing periodic boundary conditions), as we did for the incompressible and compressible Euler equations, often comes in handy, because the torus is compact but has no boundary. All the arguments so far carry over to the whole space $\R^d$ with only a little more effort.

The story becomes very different upon the appearance of physical boundaries: Weak-strong uniqueness can fail in this situation due to the possible formation of a boundary layer in inviscid models. Let us discuss again the incompressible Euler equations, for which the usual slip boundary conditions are required:
\begin{equation*}
u\cdot n=0\hspace{0.5cm}\text{on $\partial\Omega$},
\end{equation*}
where $n$ denotes the outward unit normal to the boundary of $\Omega$. For weak solutions, in general only $u\in L^2_{loc}(\Omega\times(0,T))$, and so it makes no sense to evaluate $u$ on the boundary. Instead one defines the space $H(\Omega)$ of \term{solenoidal vectorfields}, which is the completion in the norm of $L^2(\Omega)$ of the space
\begin{equation*}
\left\{w\in C_c^\infty(\Omega;\R^d): \diverg{w}=0\right\}.
\end{equation*}
A solenoidal vector field thus satisfies the divergence-free and slip boundary conditions in a weak sense. This allows us to define the notion of admissible weak solution on $\Omega\times[0,T]$:
\begin{definition}
A vectorfield $u\in L^\infty((0,T);H(\Omega))$ is called an \term{admissible weak solution} of the incompressible Euler equations with initial data $u^0\in H(\Omega)$ if
\begin{equation}\label{weakbound}
\begin{aligned}
\int_0^\tau\int_{\Omega}\partial_t\phi(x,t)\cdot u(x,t)+\nabla\phi(x,t):(u\otimes u)(x,t)dxdt=\int_{\Omega}u(x,\tau)\cdot\phi(x,\tau)-u^0(x)\cdot\phi(x,0)dx
\end{aligned}
\end{equation} 
for almost every $\tau\in(0,T)$ and every $\phi\in C^1_c(\Omega\times[0,T];\R^d)$ with $\diverg\phi=0$, and if for almost every $\tau\in(0,T)$ we have
\begin{equation*}
\frac{1}{2}\int_{\Omega}|u(x,\tau)|^2dx\leq\frac{1}{2}\int_\Omega|u^0(x)|^2dx.
\end{equation*}
\end{definition}

In order to show weak-strong uniqueness, we would like to use the relative energy argument as for Theorem \ref{basicthm}. To do this, we would have to use the strong solution $U$ as a test function in \eqref{weakbound}. However, even a classical smooth solution of the Euler equations is only required to satisfy $v\cdot n=0$ on the boundary, whereas the tangential component may be nonzero. But the test function in \eqref{weakbound} has to be zero at the boundary in \emph{every} direction. And indeed it is possible to produce an example of failure of weak-strong uniqueness that exploits exactly this discrepancy. Before we state it, let us make the following remark:\footnote{I would like to thank Eduard Feireisl for pointing this out.} A possible remedy for the problem at hand would be to demand that \eqref{weakbound} hold for all $\phi\in C^1(\Omega\times[0,T];\R^d)$ with $\diverg\phi=0$ and $\phi\cdot n=0$ on $\partial\Omega$ (thus admitting a non-vanishing tangential component for the test function). Then weak-strong uniqueness would follow exactly as in Theorem \ref{basicthm}. We argue however that this ad hoc extension of the class of test functions would amount to begging the question in favour of weak-strong uniqueness: A boundary-value problem for a system of partial differential equations imposes the equations \emph{in the interior} of the domain, which is guaranteed by testing only against compactly supported test fields, and by requiring the boundary condition, which is encoded in the assumption $u\in H(\Omega)$ for almost every time. The further requirement that \eqref{weakbound} hold additionally for test fields with non-vanishing tangential component, however, lacks justification from the modelling perspective.

\begin{theorem}[\cite{eulerboundary}]
There exists a smooth bounded domain $\Omega\subset\R^2$, a time $T>0$, smooth data $u^0\in C^\infty(\Omega)$ with $\diverg(u^0)=0$ and $u^0\cdot n=0$ on $\partial\Omega$, a smooth solution $U\in C^\infty(\Omega\times[0,T])$ with data $u^0$, and infinitely many admissible weak solutions with the same data. 
\end{theorem}  
We refer to \cite{eulerboundary} for the proof, which is based on convex integration techniques, in particular on a construction of L.\ Sz\'ekelyhidi \cite{vortexsheet}.

The good news is that weak-strong uniqueness can be restored by requiring the weak solution not to behave ``wildly" near the boundary:

\begin{theorem}
Let $\Omega\in\R^2$ be a smooth bounded domain, $U\in C^1(\bar{\Omega}\times[0,T])$ a solution of the Euler equations with $U(\cdot,0)=u^0$, and $u\in L^\infty((0,T);H(\Omega))$ an admissible weak solution with initial data $u^0$ such that $u$ is continuous\footnote{In the original paper \cite{eulerboundary}, H\"older continuity is required. However, the H\"older condition is not necessary in the proof, which works just fine even if $u$ is merely continuous.} in a neighbourhood of $\partial\Omega\times[0,T]$. Then, $u=U$ for almost every $(x,t)\in\Omega\times(0,T)$.
\end{theorem}
\begin{proof}
We omit some details in the sequel; for the full proof, see \cite{eulerboundary}.

For simplicity we work only on the upper half-plane $\Omega=\{(x_1,x_2)\in\R: x_2>0\}$\footnote{This domain is of course not bounded, but we will ignore this issue since our problem is of a local nature.}. Let $\epsilon>0$ be so small that $u$ is continuous on $\Gamma^\epsilon:=\{(x_1,x_2)\in \bar{\Omega}: x_2<\epsilon\}$. Moreover let $\chi\in C_c^\infty((0,\infty];\R)$ be a smooth function such that $0\leq\chi\leq1$ and $\chi(s)=1$ for $s\geq1$, and set $\chi^\epsilon(s)=\chi(s/\epsilon)$. Since $U$ is divergence-free and satisfies the slip boundary condition, there exists (by Poincar\'e's Lemma) a scalar potential $\psi\in C([0,T];C^2(\bar{\Omega}))\cap C^1(\bar{\Omega}\times[0,T])$ with $\psi=0$ on $\partial\Omega$ such that $U=\nabla^\perp\psi$. Inspired by Kato's famous paper \cite{kato}, we now set
\begin{equation*}
U^\epsilon(x_1,x_2,t):=\nabla^\perp\left[\chi^\epsilon\left(x_2\right)\psi(x_1,x_2,t)\right].
\end{equation*}
Then $U^\epsilon$ is $C^1$, compactly supported and divergence-free, and thus qualifies as a test function in \eqref{weakbound}. 

Hence we can estimate, for almost every $\tau\in(0,T)$,
\begin{equation}\label{REbound}
\begin{aligned}
&\frac{1}{2}\int_\Omega|u(\tau)-U^\epsilon(\tau)|^2dx\\
=&\frac{1}{2}\int_\Omega|u(\tau)|^2dx+\frac{1}{2}\int_\Omega|U^\epsilon(\tau)|^2dx-\int_\Omega u(\tau)\cdot U^\epsilon(\tau)dx\\
\leq&\frac{1}{2}\int_\Omega|U^\epsilon(\tau)|^2dx-\frac{1}{2}\int_\Omega|u^0|^2dx-\int_0^\tau\int_\Omega(\partial_t U^\epsilon\cdot u+\nabla U^\epsilon:(u\otimes u))dxdt.
\end{aligned}
\end{equation}
We now want to let $\epsilon\searrow 0$ in order to obtain an estimate for the relative energy between $u$ and $U$. To this end, observe that
\begin{equation}\label{corrector}
\begin{aligned}
U^\epsilon(x,\tau)=&\nabla^\perp\left[\chi^\epsilon\left(x_2\right)\psi(x_1,x_2,\tau)\right]\\
=&\chi^\epsilon(x_2)\nabla^\perp\psi(x_1,x_2,t)+\frac{1}{\epsilon}\psi(x_1,x_2,\tau)\chi'\left(\frac{x_2}{\epsilon}\right)e_1\\
=&\chi^\epsilon(x_2)U(x,\tau)+\frac{1}{\epsilon}\psi(x_1,x_2,\tau)\chi'\left(\frac{x_2}{\epsilon}\right)e_1.
\end{aligned}
\end{equation}
The first term clearly converges strongly in $L^2(\Omega)$ to $U$, uniformly in time, as $\epsilon\searrow 0$. For the second term, recall that $\psi$ is twice continuously differentiable in space, uniformly in time, and assumes the value zero at $x_2=0$, so that there exists a constant $C$ such that
\begin{equation*}
|\psi(x,\tau)|\leq Cx_2\leq C\epsilon
\end{equation*} 
in the boundary region $\Gamma^\epsilon$. Together with the observation that $\chi'(\cdot/\epsilon)$ vanishes outside $\Gamma^\epsilon$, we obtain convergence in $L^2(\Omega)$ of the second term of \eqref{corrector} to zero, uniformly in time. Hence we get $U^\epsilon\to U$ in $L^2(\Omega)$ as $\epsilon\searrow 0$, uniformly in time. In a similar way one can see  $\partial_tU^\epsilon\to \partial_tU$ in $L^2(\Omega)$ as $\epsilon\searrow 0$, uniformly in time.

The more problematic term is $\int_0^\tau\int_\Omega \nabla U^\epsilon:(u\otimes u) dxdt$. We split it as
\begin{equation*}
\int_0^\tau\int_\Omega \partial_1 U_1^\epsilon u_1^2 dxdt+\int_0^\tau\int_\Omega \partial_1 U_2^\epsilon u_1u_2 dxdt+\int_0^\tau\int_\Omega \partial_2 U_1^\epsilon u_1u_2 dxdt+\int_0^\tau\int_\Omega \partial_2 U_2^\epsilon u_2^2 dxdt.
\end{equation*}
In view of \eqref{corrector}, we expect the third integral to behave worst, so we skip the other integrals. We estimate it in the following way, again using \eqref{corrector}:
\begin{equation*}
\begin{aligned}
&\left|\int_0^\tau\int_\Omega \partial_2 (U_1^\epsilon-U_1) u_1u_2 dxdt\right|\\
\leq&\left|\int_0^\tau\int_\Omega (\chi^\epsilon(x_2)-1)\partial_2 U_1 u_1u_2 dxdt\right|+\frac{1}{\epsilon}\left|\int_0^\tau\int_\Omega \partial_2\psi(x,\tau)\chi'\left(\frac{x_2}{\epsilon}\right) u_1u_2 dxdt\right|\\
&\hspace{0.4cm}+\frac{1}{\epsilon^2}\left|\int_0^\tau\int_\Omega \psi(x,\tau)\chi''\left(\frac{x_2}{\epsilon}\right) u_1u_2 dxdt\right|.\\
\end{aligned}
\end{equation*}
The first term clearly converges to zero. For the second term, observe that $\chi'$ is supported on $\Gamma^\epsilon$, where by assumption $u$ is continuous; therefore, $u_1$, $\nabla\psi$, and $\chi'(\cdot/\epsilon)$ are bounded on the domain of integration, and by the slip boundary condition $u(x_1,x_2,t)\to0$ as $x_2\to0$, uniformly in $x_1, t$. It follows that the second integral is bounded by
\begin{equation*}
C\frac{1}{\epsilon}\int_{\Gamma^\epsilon}u_2(x_1,x_2,t)dxdt\to0.
\end{equation*} 
Finally, the third integral can be estimated similarly, because $\chi''(\cdot/\epsilon)$ is again supported on $\Gamma^\epsilon$, and there we have as above $|\psi(x_1,x_2,t)|\leq C\epsilon$.

We have thus shown that taking the limit $\epsilon\searrow0$ in \eqref{REbound} yields
\begin{equation*}
\frac{1}{2}\int_\Omega|u(\tau)-U(\tau)|^2dx
\leq-\int_0^\tau\int_\Omega(\partial_t U\cdot u+\nabla U:(u\otimes u))dxdt,
\end{equation*}
which allows us to continue exactly as in the proof of Theorem \ref{basicthm}.
\end{proof}

\section{An Alternative Approach}\label{alternative}
We present here an alternative proof of weak-strong-uniqueness for the incompressible Euler equations, inspired by arguments used in the context of so-called statistical solutions \cite{fjordholmetal}\footnote{I would like to thank U.\ S.\ Fjordholm and S.\ Mishra for valuable discussions on this subject.}. This approach may be of some interest, as it demonstrates the usefulness of doubling of variables techniques for the Euler equations, and avoids the assumption of time differentiability of the strong solution. We work again on the torus $\T^d$.

Write the Euler equations in components,
\begin{equation*}
\partial_t u^i(x)+\sum_{j=1}^d\partial_ju^i(x) u^j(x)+\partial_i p(x)=0.
\end{equation*}
If $U$ is the strong solution, we have similarly (but evaluated at another point $y\in\T^d$)
\begin{equation*}
\partial_t U^i(y)+\sum_{j=1}^d\partial_jU^i(y) U^j(y)+\partial_i P(y)=0.
\end{equation*}
Multiplying the first equation by $U^i(y)$, the second by $u^i(x)$, and adding everything together, we arrive at
\begin{equation}\label{correlation}
\begin{aligned}
\partial_t(u(x)\cdot U(y))&+\diverg_x((u(x)\cdot U(y))u(x))+\diverg_y((u(x)\cdot U(y))U(y))\\
+&\diverg_x(p(x)U(y))+\diverg_y(P(y)u(x))=0.
\end{aligned}
\end{equation}
Thanks to its divergence structure, this last equation admits a formulation in the sense of distributions, and it can be shown rigorously (as in Lemma 3.1 in \cite{fjordholmetal}) that \eqref{correlation} holds in such a weak sense if $u,U$ are weak solutions of the Euler equations. 

Therefore, we have for any $\phi\in C^1(\T^d\times\T^d\times[0,T])$
\begin{equation}\label{doubling}
\begin{aligned}
\frac{1}{2}&\int_0^T\int_{\T^d}\int_{\T^d}\partial_t\phi(x,y,t)|u(x)-U(y)|^2dxdydt\\
=&\frac{1}{2}\int_0^T\int_{\T^d}\int_{\T^d}\partial_t\phi(x,y,t)|u(x)|^2dxdydt+\frac{1}{2}\int_0^T\int_{\T^d}\int_{\T^d}\partial_t\phi(x,y,t)|U(y)|^2dxdydt\\
&\hspace{1cm}-\int_0^T\int_{\T^d}\int_{\T^d}\partial_t\phi(x,y,t)u(x)\cdot U(y)dxdydt\\
=&\frac{1}{2}\int_0^T\int_{\T^d}\int_{\T^d}\partial_t\phi(x,y,t)|u(x)|^2dxdydt+\frac{1}{2}\int_0^T\int_{\T^d}\int_{\T^d}\partial_t\phi(x,y,t)|U(y)|^2dxdydt\\
&\hspace{1cm}+\int_0^T\int_{\T^d}\int_{\T^d}(\nabla_x\phi(x,y,t)\cdot u(x))(u(x)\cdot U(y))dxdydt\\
&\hspace{1cm}+\int_0^T\int_{\T^d}\int_{\T^d}(\nabla_y\phi(x,y,t)\cdot U(y))(U(y)\cdot u(x))dxdydt\\
&\hspace{1cm}-\int_0^T\int_{\T^d}\int_{\T^d}\nabla_x\phi(x,y,t)\cdot p(x)U(y)dxdydt\\
&\hspace{1cm}-\int_0^T\int_{\T^d}\int_{\T^d}\nabla_y\phi(x,y,t)\cdot P(y)u(x)dxdydt.\\
\end{aligned}
\end{equation} 
For the rest of the argument, assume in addition that $U$ is $C^1$ in space and that $u$ is globally admissible. Choose 
\begin{equation}\label{testchoice}
\phi(x,y,t)=\eta^\epsilon\left(\frac{x-y}{2}\right)\chi(t)
\end{equation}  
for a standard mollifier $\eta$, $\eta^\epsilon(x)=\epsilon^{-d}\eta(x/\epsilon)$, and some $\chi\in C^1([0,T])$. It is easy to see (by the Lebesgue Differentiation Theorem) that the left hand side of \eqref{doubling} converges to
\begin{equation*}
\frac{1}{2}\int_0^T\int_{\T^d}\chi'(t)|u(x)-U(x)|^2dxdt
\end{equation*} 
as $\epsilon\searrow 0$.

for the second last line of \eqref{doubling}, observe that $\nabla_x\eta^\epsilon\left(\frac{x-y}{2}\right)=-\nabla_y\eta^\epsilon\left(\frac{x-y}{2}\right)$ and therefore, by the divergence-free condition on $U$, for the particular choice \eqref{testchoice}, 
\begin{equation*}
\int_0^T\int_{\T^d}\int_{\T^d}\nabla_x\phi(x,y,t)\cdot p(x)U(y)dxdydt=0,
\end{equation*} 
and similarly 
\begin{equation*}
\int_0^T\int_{\T^d}\int_{\T^d}\nabla_y\phi(x,y,t)\cdot P(y)u(x)dxdydt=0.
\end{equation*}
For the second and third integrals of \eqref{doubling}, we compute
\begin{equation}\label{intstep}
\begin{aligned}
&\int_0^T\int_{\T^d}\int_{\T^d}\nabla_x\phi(x,y,t)\cdot u(x))(u(x)\cdot U(y)+\nabla_y\phi(x,y,t)\cdot U(y))(U(y)\cdot u(x)dxdydt\\
&=\int_0^T\chi(t)\int_{\T^d}\int_{\T^d}(u(x)\cdot U(y))\nabla_y\eta^\epsilon\left(\frac{x-y}{2}\right)\cdot \left(U(y)-u(x)\right)dxdydt.
\end{aligned}
\end{equation}
Now we obtain on the one hand (owing to the divergence condition)
\begin{equation*}
\begin{aligned}
&\int_0^T\chi(t)\int_{\T^d}\int_{\T^d}(u(x)\cdot U(y))\nabla_y\eta^\epsilon\left(\frac{x-y}{2}\right)\cdot U(y)dxdydt\\
&=-\int_0^T\chi(t)\int_{\T^d}\int_{\T^d}\eta^\epsilon\left(\frac{x-y}{2}\right)u(x)\cdot\nabla U(y) U(y)dxdydt\\
\end{aligned}
\end{equation*}
and on the other hand
\begin{equation*}
\begin{aligned}
&-\int_0^T\chi(t)\int_{\T^d}\int_{\T^d}(u(x)\cdot U(y))\nabla_y\eta^\epsilon\left(\frac{x-y}{2}\right)\cdot u(x)dxdydt\\
&=\int_0^T\chi(t)\int_{\T^d}\int_{\T^d}\eta^\epsilon\left(\frac{x-y}{2}\right)u(x)\cdot\nabla U(y) u(x)dxdydt,
\end{aligned}
\end{equation*}
so that \eqref{intstep} turns into
\begin{equation*}
\int_0^T\chi(t)\int_{\T^d}\int_{\T^d}\eta^\epsilon\left(\frac{x-y}{2}\right)u(x)\cdot\nabla U(y) (u(x)-U(y))dxdydt
\end{equation*}
and, after another invocation of the identity $\int U\cdot \nabla U wdy=0$ for divergence-free $w$, into 
\begin{equation*}
\int_0^T\chi(t)\int_{\T^d}\int_{\T^d}\eta^\epsilon\left(\frac{x-y}{2}\right)(u(x)-U(y))\cdot\nabla_{sym} U(y) (u(x)-U(y))dxdydt.
\end{equation*}
Now again we can apply Lebesgue's Differentiation Theorem to obtain in the limit $\epsilon\searrow 0$
\begin{equation*}
\int_0^T\chi(t)\int_{\T^d}\int_{\T^d}(u(x)-U(x))\cdot\nabla U(x) (u(x)-U(x))dxdt.
\end{equation*}
To conclude the analysis of equality \eqref{doubling}, we observe that the energy admissibility of $u$ and $U$ implies that the first line on the right hand side of \eqref{doubling} is non-negative for any choice of $\chi\geq0$. 

Putting everything together, we get from \eqref{doubling}
\begin{equation*}
\frac{d}{dt}E_{rel}\leq \left|\int_{\T^d}\int_{\T^d}(u(x)-U(x))\cdot\nabla U(x) (u(x)-U(x))dx\right|\leq \norm{\nabla_{sym} U}_\infty E_{rel},
\end{equation*}
where the relative energy $E_{rel}$ is defined as in the proof of Theorem \ref{basicthm}. Hence, Gr\"onwall's inequality yields once more weak-strong uniqueness as long as $U$ is $C^1$ in space for almost every time, and $\nabla_{sym} U\in L^1((0,T);L^\infty(\T^d))$. Note that no further time regularity is required.  

\begin{remark}
In both proofs of weak-strong uniqueness for the Euler equations, the assumption on $\nabla_{sym}U$ is crucial. It is a major open problem whether uniqueness holds for the Euler equations when the solutions have less than one derivative. In the light of recent progress on Onsager's Conjecture, one conjecture is that weak solutions are unique if they belong to $C^\alpha$ with $\alpha>1/3$. However the regime $1/3<\alpha<1$ is currently completely open in terms of uniqueness. 
\end{remark}

\end{document}